\documentclass{amsart}
\usepackage[margin=1in]{geometry}

\usepackage{amsmath}
\usepackage{amsthm}
\usepackage{amsfonts}
\usepackage{amssymb}
\usepackage{eucal}
\usepackage{bm}
\usepackage[all,cmtip]{xy}
\usepackage{hyperref}
\usepackage{color}
\usepackage{tikz-cd} 
\usepackage{multicol}
\usepackage{tikz, subfigure}
\usetikzlibrary{arrows, patterns}
\usepackage[percent]{overpic}
\usepackage{graphicx, caption}

\theoremstyle{plain}
\newtheorem{theorem}{Theorem}[section]
\newtheorem{lemma}[theorem]{Lemma}
\newtheorem{proposition}[theorem]{Proposition}

\theoremstyle{definition}
\newtheorem{definition}[theorem]{Definition}
\newtheorem{example}[theorem]{Example}
\newtheorem{remark}[theorem]{Remark}
\newtheorem{question}[theorem]{Question}
\newtheorem{conjecture}[theorem]{Conjecture}
\newtheorem*{notation*}{Notation}

\newcommand{\s}{\mathfrak{s}}
\newcommand{\tspin}{\mathfrak{t}}

\DeclareMathOperator{\hf}{HF}

\DeclareMathOperator{\HF}{HF}

\DeclareMathOperator{\cfk}{CFK}
\DeclareMathOperator{\cf}{CF}
\DeclareMathOperator{\CF}{CF}

\DeclareMathOperator{\hfkhat}{\widehat{HFK}}

\DeclareMathOperator{\hfhat}{\widehat{HF}}

\newcommand\blfootnote[1]{%
  \begingroup
  \renewcommand\thefootnote{}\footnote{#1}%
  \addtocounter{footnote}{-1}%
  \endgroup
}

\begin{document}
\title{Truncated Heegaard Floer Homology and Knot Concordance Invariants}
\author{Linh Truong}
\address{Department of Mathematics, Columbia University, New York, NY 10027}
\email{ltruong@math.columbia.edu}
\maketitle
\begin{abstract} 
In this paper we construct a sequence of integer-valued concordance invariants $\nu_n(K)$ that generalize the Ozsv\'ath-Szab\'o $\nu$-invariant and the Hom-Wu $\nu^+$-invariant. 
\end{abstract}

\blfootnote{This work was partially supported by an NSF postdoctoral fellowship, DMS-1606451.}

\section{Introduction}
Ozsv\'ath and Szab\'o defined an invariant $\nu(K)$ in \cite{OzSzRational} using maps on Heegaard Floer homology $\hfhat$ induced by the two-handle cobordism corresponding to integral surgery along $K$. Hom and Wu \cite{Hom-Wu} defined $\nu^+(K)$ using maps on $\hf^+$ and showed that $\nu^+(K)$ produces arbitrarily better four-ball genus bounds than $\nu(K)$.

Motivated by the constructions of $\nu(K)$ and $\nu^+(K)$, we construct a sequence of knot invariants $\nu_n(K)$, $n \in \mathbb{Z}$, with the following properties:
\begin{enumerate}
\item[$\bullet$] $\nu_n(K)$ is a concordance invariant. 
\item[$\bullet$] $\nu_1(K) = \nu(K)$. 
\item[$\bullet$] $\nu_n(K) \leq \nu_{n+1}(K)$.
\item[$\bullet$] For sufficiently large $n$, $\nu_n(K) = \nu^+(K)$. 
\item[$\bullet$] $\nu_n(-K) = - \nu_{-n}(K)$, where $-K$ is the mirror of $K$. 
\item[$\bullet$] $\nu_n(K) \leq g_4(K)$.
\end{enumerate}

The invariants $\nu_n(K)$ are constructed from maps on truncated Heegaard Floer homology $\hf^n$. By an extension of the large integer surgery formulas to truncated Floer homology (see Propositions \ref{negsurgeryformula} and \ref{possurgeryformula}) the invariants $\nu_n(K)$ can be computed from the $\mathbb{Z}\oplus \mathbb{Z}$-filtered knot Floer chain complex $\cfk^\infty(K)$. 

Homologically thin knots are a special class of knots whose knot Floer homology lies in a single $\delta = A - M$ grading, where $A$ is the Alexander grading and $M$ is the Maslov grading. We prove that $\nu_n(K)$ of thin knots only depends on $\tau(K)$. 
\begin{theorem} \label{thinknots}
Let $K$ be a homologically thin knot with $\tau(K) = \tau$.  
\begin{enumerate}
\item[(i)] If $\tau = 0$, $\nu_n(K) = 0$ for all $n$.
\item[(ii)] 
If $\tau > 0$, 
  \[
    \nu_n(K) = \begin{cases}
        0, & \text{for } n \leq -(\tau+1)/2,\\
        \tau + 2n + 1, & \text{for } -\tau/2 \leq n \leq -1,\\
        \tau, & \text{for } n \geq 0.
        \end{cases}
  \]
\item[(iii)] If $\tau < 0$, 
  \[
    \nu_n(K) =  \begin{cases}
        \tau, & \text{for } n \leq 0,\\
        \tau + 2n - 1, & \text{for } 1 \leq n \leq -\tau/2,\\
        0, & \text{for } n \geq (-\tau+1)/2.
        \end{cases}
  \]
\end{enumerate}
\end{theorem}

The computation of $\nu_n(K)$ for thin knots illustrates that the gap between $\nu_n(K)$ and $\nu_{n+1}(K)$ can be more than one. 
In fact, the gap between $\nu_n(K)$ and $\nu_{n+1}(K)$ can be arbitrarily big.

\begin{theorem}
Let $T_{p, p+1}$ denote the $(p, p+1)$ torus knot. For $p > 3$,
\begin{eqnarray*}
\nu_{-1}(T_{p, p+1}) - \nu_{-2}(T_{p,p+1}) = p .
\end{eqnarray*}
\end{theorem}
\vspace{3mm}
\noindent{}\textbf{Organization of the paper.} In Section \ref{section-background} we review the  constructions of the concordance invariants $\nu(K)$ and $\nu^+(K)$. In Section \ref{section-invariants} we define the invariants $\nu_n(K)$ and prove its properties: monotonicity, stabilization, and behavior under mirroring. In Section \ref{section-computations} we compute $\nu_n(K)$ for special families of knots and compare them to $\nu(K)$ and $\nu^+(K)$. In Section \ref{section-future} we pose some questions about the concordance invariants $\nu_n(K)$. 
\\\\
\noindent{}\textbf{Acknowledgements.} The author thanks her advisors Peter Ozsv\'ath and Zolt\'an Szab\'o for their guidance. She also thanks Adam Levine for carefully reading the thesis version and for corrections. 

\section{A brief background on the invariants $\nu(K)$ and $\nu^+(K)$} \label{section-background}
A four-dimensional cobordism equipped with a Spin$^c$ structure between two three-manifolds induces a map on the Heegaard Floer homology groups \cite{OzSzfour}. In particular, for a knot $K$ in $S^3$, the 2-handle attachment cobordism from $S^3_{N}(K)$, respectively $S^3_{-N}(K)$, to $S^3$ induces maps:
\begin{eqnarray}
\label{eqn}
\widehat{v}_{s, *}: \widehat{\hf}(S^3_{N}(K), \mathfrak{s}_s) \to \widehat{\hf}(S^3),
& \text{ respectively, } & 
\widehat{v}_{s, *}': \widehat{\hf}(S^3) \to \widehat{\hf}(S^3_{-N}(K), \mathfrak{s}_s), \\
\label{pluseqn}
v_{s, *}^+:{\hf}^+(S^3_{N}(K), \mathfrak{s}_s) \to {\hf}^+(S^3),
& \text{ respectively, } &
{v}_{s, *}^{+'}: {\hf}^+(S^3) \to {\hf}^+(S^3_{-N}(K), \mathfrak{s}_s), \\
\label{minuseqn}
v^-_{s, *} : \hf^-(S^3_N K, \s_s) \to \hf^-(S^3),
& \text{respectively, } &
v_{s, *}^{-'}: \hf^-(S^3) \to \hf^-(S^3_{-N} K, \s_s), 
\end{eqnarray}
where $\mathfrak{s}_s$ denotes the restriction to $S^3_N(K)$, respectively $S^3_{-N}(K)$, of a Spin$^c$ structure $\mathfrak{t}$ on the corresponding 2-handle cobordism such that 
\begin{eqnarray*}
\langle c_1(\tspin), [\widehat{F}] \rangle + N = 2s, 
&\text{ respectively, }&
\langle c_1(\tspin), [\widehat{F}] \rangle - N = 2s, 
\end{eqnarray*}
where $\widehat{F}$ is a capped-off Seifert surface for $K$. These cobordism maps on $\widehat{\hf}$ and $\hf^+$ play a key role in defining the concordance invariants $\nu$ and $\nu^+$.

\begin{definition}[\cite{OzSzRational}, Section 9]
The concordance invariant $\nu(K)$ is defined as
\[\nu (K) = \min\{s \in \mathbb{Z} \ | \ \widehat{v}_{s, *} \text{ is surjective}\} .\]
\end{definition}

\begin{definition}
The concordance invariant $\nu'(K)$ is defined as 
\[ \nu'(K) = \max\{ s \in \mathbb{Z} \ | \ \widehat{v}_{s, *}' \text{ is injective} \} .\]
\end{definition}

For a rational homology $3$-sphere $Y$ with a Spin$^c$ structure $\s$, $\hf^+(Y, \s)$ can be decomposed as the direct sum of two groups: the first group is the image of $\hf^\infty(Y, \s) \cong \mathbb{F}[U, U^{-1}]$ in $\hf^+(Y, \s)$, which is isomorphic to $\mathfrak{T}^+ = \mathbb{F}[U, U^{-1}]/U\mathbb{F}[U]$; the second group is $\hf_{\text{red}}(Y, \s) = \hf^+(Y, \s)/\mathfrak{T}^+ $. That is, 
\[ \hf^+(Y, \s) = \mathfrak{T}^+ \oplus \hf_\text{red}(Y, \s).\]

\begin{definition}[\cite{Hom-Wu}]
The concordance invariant $\nu^+$ is defined as
\[\nu^+(K) = \min\{s \ | \ v_{s,*}^+: \hf^+(S^3_N(K), \s_s) \to \hf^+(S^3) \text{ sends 1 to 1}\}\]
where $1$ denotes the lowest graded generator in the subgroup $\mathfrak{T}^+$ of the homology, and $N$ is sufficiently large so that the integer surgery formula holds. 
\end{definition}

Equivalently, Hom \cite{hom-survey} defines the invariant $\nu^-(K)$ in terms of the map $v^-_{s, *} : \hf^-(S^3_N K, \s_s) \to \hf^-(S^3)$.
\begin{definition}[\cite{hom-survey}] The concordance invariant
\[\nu^-(K) = \text{min}\{ s \in \mathbb{Z} \ | \ v^-_{s, *} \text{ is surjective}\} \]
is equal to $\nu^+(K)$.
\end{definition}

Hom and Wu show that 
\[ \tau(K) \leq \nu(K) \leq \nu^+(K)\]
and $\nu^+(K) \geq 0$. In addition, $\nu^+$ gives a lower bound on the four-ball genus
$\nu^+(K) \leq g_4(K)$.
Furthermore, Hom and Wu provide a family of knots where $\nu^+(K)$ is an arbitrarily better bound on $g_4(K)$ than $\tau(K)$.

The concordance invariants $\nu$ and $\nu^+$ are easily computable from $\cfk^\infty(K)$ via the large integer surgery formulas. 
Let $CX$ denote the subgroup of $\cfk^\infty(K)$ generated by elements $[x,i,j]$ that lie in filtration level $(i,j) \in X \subset \mathbb{Z} \oplus \mathbb{Z}$. 
Consider the chain maps
\begin{eqnarray*}
\widehat{v}_s: C\{\max(i,j-s) = 0\} \to C\{i=0\},
\\
{v}_s^+: C\{\max(i,j-s) \geq 0\} \to C\{i \geq 0\},
\end{eqnarray*}
defined by taking the quotient by $C\{i < 0 , j = s\}$, respectively $C\{ i < 0, j \geq s\}$, followed by the inclusions.
The large integer surgery formula of Ozsv\'ath-Szab\'o \cite{HFK} asserts that the maps $\widehat{v}_s$ and ${v}_s^+$ induce the maps from \eqref{eqn} and \eqref{pluseqn}. Similarly, consider the chain maps
\begin{eqnarray*}
\widehat{v}_s':  C\{i=0\} \to C\{\min(i,j-s) = 0\},
\\
{v}_s^{+'}:   C\{i \geq 0\} \to C\{\min(i,j-s) \geq 0\},
\end{eqnarray*}
consisting of quotienting by $C\{ i = 0, j \leq s \}$ followed by the inclusion. Ozsv\'ath and Szab\'o \cite{HFK} show that these maps induce the maps from \eqref{eqn} and \eqref{pluseqn}.

We introduce a concordance invariant $\nu^{+'}$, so that the pair $\nu^+$ and $\nu^{+'}$ is the $\hf^+$ analogue to the pair $\nu$ and $\nu'$.
\begin{definition} The concordance invariant $\nu^{+'}$ is defined as 
\[ \nu^{+'}(K) = \text{max}\{ s  \in \mathbb{Z}\ | \ v_{s,*}^{+'} : \hf^+(S^3) \to \hf^+(S^3_{-N}(K), \s_s) \text{ is injective}\} \]
where $-N$ is sufficiently negative so that the (negative) large integer surgery formula holds.
\end{definition}
We prove a mirroring property which relates $\nu^{+'}(K)$ to the invariant $\nu^+(-K)$ of the mirror of $K$:
\begin{lemma}
$\nu^{+'}(K) = -\nu^+(-K)$.
\label{lemma:nuplus}
\end{lemma}
\begin{proof}
Recall the symmetry of $\cfk^\infty$ under mirroring (\cite{HFK}, Section 3.5):
\[\cfk^\infty(-K) \simeq \cfk^\infty(K)^*\]
where $\cfk^\infty(K)^*$ is the dual complex $\text{Hom}_{\mathbb{F}[U, U^{-1}]}(\cfk^\infty(K), \mathbb{F}[U, U^{-1}])$. Therefore, the following conditions are equivalent:
\begin{eqnarray*}
&&v_{-s,*}^{+'} \colon \hf^+(S^3) \to \hf^+(S^3_{-N}(K), \s_{-s}) \text{ is injective}
\\
&\iff&
v_{s,*}^- \colon \hf^-(S^3_{N}(-K), \s_{s}) \to \hf^-(S^3) \text{ is surjective}
\end{eqnarray*}
which implies the result. 
\end{proof}

It follows from the above lemma that the invariant $\nu^{+'}$ exhibits properties similar to $\nu^+$:
\[\nu^{+'}(K) \leq \nu'(K) \leq \tau(K) \leq \nu(K) \leq \nu^+(K)\]
and $\nu^{+'}(K) \leq 0$. In addition, the absolute value of $\nu^{+'}(K)$ gives a lower bound on the four-ball genus. 
\begin{theorem}
$| \nu^{+'}(K) | \leq g_4(K)$
\end{theorem}
\begin{proof}
This follows from the fact that $\nu^+(K) \leq g_4(K)$ and Lemma~\ref{lemma:nuplus}.
\end{proof}

\section{The concordance invariants $\nu_n(K)$} \label{section-invariants}

The construction of the concordance invariants $\nu_n(K)$ uses truncated Heegaard Floer homology $\HF^n(Y, \s)$, described in \cite{verticaltruncation, OzSzintegersurgeries}.
$\HF^n(Y, \s)$ is the homology of the kernel $\cf^n(Y, \s)$ of the multiplication map \[U^n: \cf^+(Y, \s) \to \cf^+(Y, \s)\] where $n \in \mathbb{Z}_+$.  
The two-handle cobordism from $S^3_N K$, respectively $S^3_{-N}K$, to $S^3$ induce a map on the truncated Floer chain complex
\begin{eqnarray*}
v_{s}^n : \CF^n(S^3_NK, \mathfrak{s}_s) \to \CF^n(S^3), 
& \text{ respectively, } & 
v_{s}^{-n}: \CF^n(S^3) \to \CF^n(S^3_{-N}K, \s_s),
\end{eqnarray*}
and on the truncated Floer homology 
\begin{eqnarray*}
v_{s, *}^n : \hf^n(S^3_NK, \mathfrak{s}_s) \to \hf^n(S^3),
& \text{ respectively, } &
v_{s, *}^{-n}: \hf^n(S^3) \to \hf^n(S^3_{-N}K, \s_s).
\end{eqnarray*}
where $\mathfrak{s}_s$ denotes the restriction to $S^3_N(K)$, respectively $S^3_{-N}(K)$, of a Spin$^c$ structure $\mathfrak{t}$ on the corresponding 2-handle cobordism such that 
\begin{eqnarray*}
\langle c_1(\tspin), [\widehat{F}] \rangle + N = 2s, 
&\text{ respectively, }&
\langle c_1(\tspin), [\widehat{F}] \rangle - N = 2s, 
\end{eqnarray*}
where $\widehat{F}$ is a capped-off Seifert surface for $K$. These cobordism maps on $\widehat{\hf}$ and $\hf^+$ play a key role in defining the concordance invariants $\nu$ and $\nu^+$. 

We extend the large integer surgery formula of Ozsv\'ath-Szab\'o to truncated Heegaard Floer homology. 
\begin{proposition}
 [Negative Large Integer Surgery Formula for $\hf^n$]
 \label{negsurgeryformula}
Consider the subquotient complex 
\[\cfk^{-n}(S^3, K,m) = C\{ 0 \leq \min(i,j-m) \leq n-1\}\]
of $\cfk^+(S^3, K, m)$ generated by $[\textbf{x},i,j]$ with $0 \leq \min(i,j-m) \leq n-1$. 
For each $m \in \mathbb{Z}$, there is an integer $N(m)$ such that for all $p \geq N(m)$, the map $\Phi$ of Ozsvath-Szabo induces isomorphisms in the following diagram:
\[
\begin{tikzcd}
0 \arrow{r}{} & \cfk^{-n}(S^3,K,m)   \arrow{r}{} \arrow{d}{\Phi(n)} & \cfk^+(S^3,K,m) \arrow{r}{U^n}\arrow{d}{\Phi^+} & \cfk^+(S^3,K,m) \arrow{r}{} \arrow{d}{\Phi^+}&0 \\
0 \arrow{r}{}  & \CF^n(S^3_{-p}K,[m])  \arrow{r}{} &  \CF^+(S^3_{-p}K,[m]) \arrow{r}{U^n} &\CF^+(S^3_{-p}K,[m]) \arrow{r}{} & 0 .\\
\end{tikzcd}
\]
\end{proposition}
\begin{proof}
The map $\Phi^+$ is an isomorphism of chain complexes by Theorem 4.1 of \cite{HFK}. By the Five Lemma, so is $\Phi(n)$. 
\end{proof}

\begin{proposition}[Large Positive Integer Surgery Formula for $\hf^n$] 
\label{possurgeryformula}
Consider the subquotient complex
\[\cfk^n(S^3, K, m) = C\{ 0 \leq \max(i,j-m) \leq n-1\},\]
of $\cfk^+(S^3,K,m)$ generated by $[\textbf{x},i,j]$ with $0 \leq \max(i,j-s) \leq n-1$. For each $m \in \mathbb{Z}$, there is an integer $N(m)$ such that for all $p \geq N(m)$, the map $\Psi$ of Ozsvath-Szabo induces isomorphisms in the following diagram:
\[
\begin{tikzcd}
0 \arrow{r}{} & \CF^n(S^3_p K, [m])   \arrow{r}{} \arrow{d}{\Psi(n)} & \cfk^+(S^3_pK, [m]) \arrow{r}{U^n}\arrow{d}{\Psi^+} & \cfk^+(S^3_pK, [m]) \arrow{r}{} \arrow{d}{\Psi^+}&0 \\
0 \arrow{r}{}  & \cfk^n(S^3,K,m)  \arrow{r}{} &  ^{b}\cfk^+(S^3,K,m) \arrow{r}{U^n} &^{b}\cfk^+(S^3,K,m) \arrow{r}{} & 0 .\\
\end{tikzcd}
\]
\end{proposition}
\begin{proof}
The map $\Psi^+$ is an isomorphism of chain complexes by Theorem 4.4 of \cite{HFK}. By the Five Lemma, so is $\Psi(n)$.
\end{proof}

\begin{notation*}We sometimes use the notation for $n > 0$:
\begin{eqnarray*}
A^n_m &=& C\{ 0 \leq \max(i,j-m) \leq n-1\}
\\
A^{-n}_m &=& C\{ 0 \leq \min(i,j-m) \leq n-1\} 
\\
A_m^+ &=&C\{ 0 \leq \max(i,j-m) \} 
\\
A^{+'}_m &=&C\{ 0 \leq \min(i,j-m)\} 
\end{eqnarray*}
and 
\begin{eqnarray*}
B^{n} &=& C\{0 \leq i \leq n-1\}
\\
B^+ &=& C\{0 \leq i\}.
\end{eqnarray*}
\end{notation*}

\label{section-mu-n}

The cobordism maps on truncated Heegaard Floer groups lead us to define concordance invariants $\nu_n$.
\begin{definition} 
For $n > 0$, define

\begin{eqnarray*}
\nu_n(K) = \text{min}\{ s \in \mathbb{Z}\ | \ v_{s}^n : \CF^n(S^3_N(K), s) \to \CF^n(S^3)  \text{ induces a surjection on homology}\}, 
\end{eqnarray*}
and for $n < 0$, define
\begin{eqnarray*}
\nu_n(K) = \text{max}\{ s  \in \mathbb{Z}\ | \ v_{s}^{n} : \CF^{-n}(S^3) \to \CF^{-n}(S^3_{-N}(K), s)  \text{ induces an injection on homology}\}, 
\end{eqnarray*}
where $N$ is sufficiently large so that the Ozsv\'ath-Szab\'o large integer surgery formula of \cite{HFK} holds. For $n=0$, we define $\nu_0(K) = \tau(K)$. 
\end{definition}

\begin{remark}
For $n = \pm 1$, these invariants are already known as: 
$\nu_1(K) = \nu(K) $ and $\nu_{-1}(K) = \nu'(K)$.
\end{remark}

\begin{proposition}  $\nu_n(K)$ is a concordance invariant. 
\begin{proof} Suppose $K_1$ is concordant to $K_2$. Then $S^3_{N}(K_1)$ is homology cobordant to $S^3_{N}(K_2)$. This implies that there exists a (smooth, connected, oriented) cobordism $W$ from $S^3_{N}(K_1)$ to $S^3_{N}(K_2)$ with $H_i(W, \mathbb{Q}) = 0$ for $i = 1, 2$. 

The map 
\[\hf^n(S^3_{N}(K_1), s) \to \hf^n(S^3)\] 
induced by the cobordism obtained by adding a two-handle along $K_1$ factors through $\hf^n(S^3_{N}(K_2), s)$. So if it is surjective then the map 
\[\hf^n(S^3_{N}(K_2), s) \to \hf^n(S^3)\]
is also surjective. So we get that $\nu_n(K_1) \geq \nu_n(K_2)$. The same argument with $K_1$ and $K_2$ switched shows the inequality $\nu_n(K_2) \geq \nu_n(K_1)$. Therefore $\nu_n(K_1) = \nu_n(K_2)$.  

For negative $n$, that $\nu_n(K)$ is a concordance invariant follows from a similar argument to the above.
\end{proof}
\end{proposition}

\begin{proposition}[Mirroring property]$\nu_n(-{K}) = - \nu_{-n}(K)$. 
\begin{proof}
Recall the symmetry of $\cfk^\infty$ under mirroring (\cite{HFK}, Section 3.5):
\[\cfk^\infty(-K) \simeq \cfk^\infty(K)^*\]
where $\cfk^\infty(K)^*$ is the dual complex $\text{Hom}_{\mathbb{F}[U, U^{-1}]}(\cfk^\infty(K), \mathbb{F}[U, U^{-1}])$. Letting $C = \cfk^\infty(S^3,K)$, and $n  > 0$, the following conditions are equivalent:
\begin{eqnarray*}
&&v_{-s,*}^{-n}(K) \colon \hf^n(S^3) \to \hf^n(S^3_{-N}(K), \s_{-s}) \text{ is injective}\\
&\iff&v_{-s}^{-n}(K) \colon C\{0 \leq i \leq n-1\} \to C\{ 0 \leq \min(i,j+s) \leq n-1\} \text{ is injective on } H_*\\
&\iff&
U^{n-1}v^n_s(-K) \colon C\{-(n-1) \leq \max(i,j-s) \leq 0\} \to C\{-(n-1 )\leq i \leq 0\} \text{ is surjective on }H_*
\\
&\iff&
v_{s}^n(-K) \colon C\{0 \leq \max(i,j-s) \leq n-1 \} \to C\{ 0 \leq i \leq n-1\}  \text{ is surjective on }H_*\\
&\iff&
v_{s,*}^n(-K) \colon \hf^n(S^3_{N}(-K), \s_{s}) \to \hf^n(S^3) \text{ is surjective}
\end{eqnarray*}
where $U^{n-1}$ is a degree-shifting isomorphism on $\cfk^\infty(K)$.
Therefore,
\begin{eqnarray*}
\nu_n(-K) = \min(s \in \mathbb{Z} \ | \ v_{s,*}^n(-K) \text{ is surjective}) = -\max( -s \in \mathbb{Z} \ | \ v_{-s,*}^{-n}(K) \text{ is injective}) = - \nu_{-n}(K).
\end{eqnarray*}
\end{proof}
\end{proposition}

\begin{proposition}[Monotonicity] $ \nu_n(K) \leq \nu_{n+1}(K)$. 
\end{proposition}
\begin{proof} It is known that $\nu_{-1}(K) \leq \tau(K) \leq \nu_1(K)$, so we focus on the two separate cases where $n > 0$ and $n < 0$.

For $n > 0$, consider the commutative diagram
\[
\begin{tikzcd}
\HF^{n+1}(S^3_NK,s) \arrow{r}{v_{s,*}^{n+1}} \arrow{d}{\cdot U} & \HF^{n+1}(S^3)  \arrow{d}{\cdot U} \\
\HF^{n}(S^3_NK,s) \arrow{r}{v_{s,*}^n}  & \HF^{n}(S^3)
\end{tikzcd}
\] 
where the vertical maps are given by multiplication by $U$. 
The vertical map on the right is surjective. Thus, if ${v_{s,*}^{n+1}}$ is surjective, then so is ${v_{s,*}^n}$.

For $n < 0$, consider the commutative diagram
\[
\begin{tikzcd}
\HF^{-n}(S^3) \arrow{r}{v_{s,*}^n} \arrow{d}{i_a'} & \HF^{-n}(S^3_{-N} K,s)\arrow{d}{i_b'} \\
\HF^{-(n-1)}(S^3) \arrow{r}{v_{s,*}^{n-1}} & \HF^{-(n-1)}(S^3_{-N}K,s) 
\end{tikzcd}
\]
where the vertical maps are induced by inclusion of chain groups. In particular, the left map $i'_a$ is injective on homology. Therefore, if $v_{s,*}^{n-1}(K)$ is injective, then so is $v_{s,*}^n(K)$. We conclude that $\nu_{n-1}(K) \leq \nu_n(K)$.
\end{proof}

\begin{proposition}[Boundedness] $\nu^{+'}(K) \leq \nu_n(K) \leq \nu^+(K)$ for all $n$. 
\begin{proof}
It is known that $\nu(K) \leq \nu^+(K)$ from \cite{Hom-Wu}. For $n \geq 1$, consider the commutative diagram
\[
\begin{tikzcd}
H_*(A_k^-) \arrow{r}{j_A} \arrow{d}{v_{k,*}^-} & H_*(A_k^n)\arrow{d}{v_{k,*}^n} \\
H_*(B^-) \arrow{r}{j_B} & H_*(B^n)
\end{tikzcd}
\]
The map $j_B$ is surjective, so if $v_{k,*}^-$ is surjective, then so is $v_{k,*}^n$. 

For $n \leq -1$, consider the commutative diagram
\[
\begin{tikzcd}
 H_*(B^n)\arrow{r}{i_B}\arrow{d}{v_{k,*}^n}  & H_*(B^+)\arrow{d}{v_{k,*}^{+'}}  \\
 H_*(A_k^n)\arrow{r}{i_A} & H_*(A_k^{+'})
\end{tikzcd}
\]
The map $i_B$ is injective, so if $v_{k,*}^{+'}$ is injective, then so is $v_{k,*}^n$. 
\end{proof}

\end{proposition}

\begin{proposition}[Stabilization] For sufficiently large positive $n$, $\nu_n(K) = \nu^+(K)$ and $\nu_{-n}(K) = \nu^{+'}(K)$. 
\end{proposition}
\begin{proof}
Let $C_1 = \CF^-(S^3_NK, s)$ and $C_2 = \CF^-(S^3)$. 
There is a canonical degree shifting isomorphism
\[\cf^n(Y, \s) \cong \cf^-(Y, \s) \otimes_{\mathbb{F}[U]} \frac{\mathbb{F}[U]}{U^n}.\]
Moreover, the map 
\[v_{n,s}^- :  C_1 \otimes_{\mathbb{F}[U]} \frac{\mathbb{F}[U]}{U^n} \to  C_2 \otimes_{\mathbb{F}[U]} \frac{\mathbb{F}[U]}{U^n}\]
is the same as the map $v_s^n: \cf^n(S^3_NK,s) \to \cf^n(S^3)$. We show that if $v^-_{s,*}$ is not surjective, then neither is $v_{s,*}^n$ for sufficiently large $n$. By the universal coefficient theorem: 
\[
\begin{tikzcd}
0  \to H_*(C_1) \otimes \frac{\mathbb{F}[U]}{U^n} \arrow{d}{v^-_{s,*} \otimes \text{id} } \arrow{r}{i_1} & H_*(C_1 \otimes \frac{\mathbb{F}[U]}{U^n})  \arrow{d}{v_{s,*}^n} \arrow{r}{} & \text{Tor}(H_*C_1, \frac{\mathbb{F}[U]}{U^n} ) \arrow{d}{\text{Tor}(v_s^-)} \to 0 \\
0  \to  H_*(C_2) \otimes\frac{\mathbb{F}[U]}{U^n} \arrow{r}{i_2} & H_*(C_2 \otimes \frac{\mathbb{F}[U]}{U^n})  \arrow{r}{} & \text{Tor}(H_*C_2, \frac{\mathbb{F}[U]}{U^n} )  \to 0
\end{tikzcd}
\]
where all tensor products are taken over $\mathbb{F}[U]$. 

We note the following facts:
\begin{enumerate}
\item[$\bullet$] For a rational homology $3$-sphere $Y$, $\hf^-(Y, \s) / \{U \text{-torsion}\} = \mathfrak{T}^- = \mathbb{F}[U]$. So $H_*(C_1) = \mathfrak{T}^- \bigoplus (\oplus \frac{\mathbb{F}[U]}{U^{m_i}})$. 
\item[$\bullet$] $\mathfrak{T}^- \otimes \frac{\mathbb{F}[U]}{U^n} =\frac{\mathbb{F}[U]}{U^n}  $ and
$\frac{\mathbb{F}[U]}{U^{m_i}} \otimes \frac{\mathbb{F}[U]}{U^n} = \frac{\mathbb{F}[U]}{U^{m_i}}$. So 
$H_*(C_1) \otimes \frac{\mathbb{F}[U]}{U^n} = \frac{\mathbb{F}[U]}{U^n} \bigoplus (\oplus \frac{\mathbb{F}[U]}{U^{m_i}})$.
\item[$\bullet$] $\text{Tor}(\frac{\mathbb{F}[U]}{U^m}, \frac{\mathbb{F}[U]}{U^n} ) =\frac{\mathbb{F}[U]}{U^m}$ if $m < n$.
\item[$\bullet$] $H_*(C_2 \otimes \frac{\mathbb{F}[U]}{U^n}) =\frac{\mathbb{F}[U]}{U^n}$. 

\end{enumerate}
Assume $n$ is sufficiently large so that $m_i < n$ for all $m_i$. So the above Tor groups are $n-1$ torsion.

If $v^-_s$ is not surjective, we can further choose $n$ sufficiently large so that the image of $v^-_{s,*} \otimes \text{id} $ is $n-1$ $U-$torsion. 
By commutativity of the diagram, the image of $v_{s,*}^n \circ i_1$ is $n-1$ $U-$torsion. 

Suppose $\xi \in H_*(C_1 \otimes \frac{\mathbb{F}[U]}{U^n})$ such that $v_{s,*}^n(\xi)$ is an element of order $n$. Then since the short exact sequence in the universal coefficient theorem splits, $\xi =  \alpha + \beta$ where $\alpha \in H_*(C_1) \otimes \frac{\mathbb{F}[U]}{U^n}$ and $\beta \in  \text{Tor}(H_*C_1, \frac{\mathbb{F}[U]}{U^n} )$. But 
\[U^{n-1}\cdot v_{s,*}^n(\alpha + \beta) = v_{s,*}^n(U^{n-1}\alpha) +v_{s,*}^n (U^{n-1} \beta) = 0.\]

Since $H_*(C_2 \otimes \frac{\mathbb{F}[U]}{U^n}) = \frac{\mathbb{F}[U]}{U^n}$, $v_{s,*}^n$ is not surjective. Therefore, for sufficiently large $n$, $\nu_n(K) = \nu^+(K)$. 

Finally, by the mirroring property, $\nu_n(K) = \nu^{+'}(K)$ for sufficiently large negative integers $n$. 
\end{proof}

The fact that $\nu_n(K)$ are not concordance homomorphisms from $\mathcal{C}$ to $\mathbb{Z}$ can easily be seen. ($\nu_n(K)$ is not additive under connected sum of knots). For $n = 1$, just consider two knots with $\varepsilon(K) =  \varepsilon(K') = -1$. Then 
\begin{eqnarray*}
\nu(K) = \tau(K) + 1 & \text{ and } & \nu(K') = \tau(K') + 1
\end{eqnarray*}
but
\begin{eqnarray*}
\nu(K\#K') = \tau(K\#K') + 1 = \tau(K) + \tau(K') + 1 < \nu(K) + \nu(K').
\end{eqnarray*}

\section{Computations}\label{section-computations}
Knot Floer homology groups can be easily computed for certain special families of knots. For example, homologically thin knots are knots with $\hfkhat(K)$ supported in a single $\delta$-grading, where $\delta = A - M$. If the homology is supported on the diagonal $\delta = -\sigma(K)/2$, where $\sigma(K)$ denotes the knot signature, then we say the knot is $\sigma$-thin. The class of $\sigma$-thin knots contains as a proper subset all quasi-alternating knots, and in particular all alternating knots. The following theorem shows that $\nu_n(K)$ of thin knots only depends on $\tau(K)$. 
\begin{theorem} \label{thinknots}
Let $K$ be a homologically thin knot with $\tau(K) = \tau$.  
\begin{enumerate}
\item[(i)] If $\tau = 0$, $\nu_n(K) = 0$ for all $n$.
\item[(ii)] 
If $\tau > 0$, 
  \[
    \nu_n(K) = \begin{cases}
        0, & \text{for } n \leq -(\tau+1)/2,\\
        \tau + 2n + 1, & \text{for } -\tau/2 \leq n \leq -1,\\
        \tau, & \text{for } n \geq 0.
        \end{cases}
  \]
\item[(iii)] If $\tau < 0$, 
  \[
    \nu_n(K) =  \begin{cases}
        \tau, & \text{for } n \leq 0,\\
        \tau + 2n - 1, & \text{for } 1 \leq n \leq -\tau/2,\\
        0, & \text{for } n \geq (-\tau+1)/2.
        \end{cases}
  \]
\end{enumerate}
\begin{proof}
In \cite{Petkova} Petkova constructs model complexes for $\cfk^\infty(K)$ of homologically thin knots. She shows if $\tau(K) = \tau$, then the model chain complex contains a direct summand isomorphic to 
\[ \cfk^\infty(T_{2,2\tau+1}) \text{ if } \tau > 0, \ \  \text{ or } \ \ \cfk^\infty(T_{2,2\tau-1}) \text{ if } \tau \leq 0.\]
This summand supports $H_*(\cfk^\infty(K))$ and thus $\nu_n(T_{2,2\tau\pm1})$ determines the values of $\nu_n(K)$. 

Without loss of generality, assume $\tau(K) > 0$. 
The chain complex $\cfk^\infty(T_{2,2\tau+1})$ is generated over $\mathbb{F}[U, U^{-1}]$ by generators
$\{z_p\}_{p=1}^{2\tau+1}$ with $U$-filtration levels $i$ and Alexander filtration levels $j$ specified by (for all $1 \leq p \leq 2\tau+1$):
\begin{equation*}
\begin{split}
j(z_p) = \begin{cases}
\tau - \frac{p - 1}{2} & \text{ if } p \text{ odd}; \\
\tau - \frac{p - 2}{2} & \text{ if } p \text{ even}.
\end{cases}
\end{split}
\qquad \qquad
\begin{split}
i(z_p) = \begin{cases}
\frac{p-1}{2} & \text{ if } p \text{ odd}; 
 \\
\frac{p}{2} & \text{ if } p \text{ even}.
\end{cases}
\end{split}
\end{equation*}
\noindent
and differential:
\begin{eqnarray*}
\partial z_{p} = 
\begin{cases}
0, & \text{ if } p \text{ odd};
\\
z_{p-1} + z_{p+1}, & \text{ if } p \text{ even}.
\end{cases}
\end{eqnarray*}
The above complex with generators $\{z_p\}_{p=1}^{2\tau+1}$ and given differential maps forms the generating ``staircase" complex $C_\tau$, and $\cfk^\infty(T_{2,2\tau+1})$ is the tensor product of this staircase complex $C_\tau \otimes_\mathbb{F} \mathbb{F}[U, U^{-1}]$. The $U$-action lowers $i$ and $j$ by one. 
\begin{itemize}
\item[$\bullet$  ] \textbf{Computation of $\nu^+(T_{2,2\tau+1})$ and $\nu_n(T_{2,2\tau+1})$ for $n > 0$}. 
Since $C\{i < 0, j \geq \tau\} = 0$, the map $v_\tau^+$ is the same as the inclusion 
\[C\{0 \leq i \leq n - 1, j \leq \tau + n -1\} \to C\{0 \leq i \leq n-1\} = B^n.\]
Moreover the generator with the highest Alexander grading in $C\{0 \leq i \leq n-1\} $ is $U^{n-1}z_1$ with 
\[j(U^{n-1}z_1) = \tau+n-1.\]
Thus, $C\{0 \leq i \leq n-1, j > \tau + n - 1\} = 0$. 
That is, the inclusion $v_\tau^+$ is an isomorphism of chain complexes, so $\nu^+(K) = \tau(K)$. 
Therefore, $\nu_n(K) = \tau(K)$ for all $n \geq 0$. 

\item[$\bullet$] \textbf{Computation of $\nu^{+'}(T_{2,2\tau+1})$}. The homology of $B^+$ is generated by $\{[U^{-i}z_1]\}$ for all $i \geq 0$. The subquotient complex $A^{+'}_0$ contains $U^{-i}C_\tau$ for all $i \geq 0$, and the homology of $U^{-i}C_\tau$ is generated by the class $[U^{-i}z_1]$. Therefore, $v^{+'}_{0,*}[U^{-i}z_1] \neq 0$ in $H_*(A^{+'}_0)$, and $v^{+'}_{0,*}$ is injective. So $\nu^{+'}(T_{2,2\tau+1}) \geq 0$. But since $\nu^{+'}(K) \leq 0$ for any knot $K$, we conclude $\nu^{+'}(T_{2,2\tau+1}) = 0$. 
\end{itemize}

\begin{itemize}
\item[$\bullet$] \textbf{Computation of $\nu_n(T_{2,2\tau+1})$ for $ -\tau/2 \leq n \leq -1$.} Consider the subquotient complex $A^n_k$ where $k = \tau + 2n + 1$. For each $1 \leq p \leq 2 \tau + 1$: 
\begin{eqnarray*}
\min(i(z_p),j(z_p)-k) =&
\begin{cases}
\min(\frac{p-1}{2}, -\frac{p-1}{2} - 2n -1) & \text{ if } p \text{ is odd};
\\
\min(\frac{p}{2}, -\frac{p-2}{2} - 2n -1)  &  \text{ if } p \text{ is even}.
\end{cases}
\\
=&
\begin{cases}
\frac{p-1}{2} & \text{ if } p \text{ is odd and }  p \leq -2n-1;
\\
-\frac{p-1}{2} - 2n - 1 & \text{ if } p \text{ is odd and }  p > -2n;
\\
\frac{p}{2}  &  \text{ if } p \text{ is even and } p \leq -2n;
\\
 -\frac{p-2}{2} - 2n -1 &  \text{ if } p \text{ is even and } p > -2n.
\end{cases}
\end{eqnarray*}
Using these formulas, it is straightforward to check that $A^n_k$ contains $z_p$ for $1 \leq p \leq -2n-1$ but $z_{-2n} \notin A^n_k$.
Therefore, $[z_1] \neq 0$ in $A^n_k$. Similarly, for $1\leq i \leq -n-1$, $A^n_k$ contains $U^{-i}z_p$ for $1 \leq p \leq -2(n+i)-1$ but $z_{-2(n+i)} \notin A^n_k$. Therefore, $[U^{-i}z_1] \neq 0$ in $H_*(A^n_k)$. Since $H_*(B^n)$ is generated by $[U^{-i}z_1]$ for $0\leq i \leq -n-1$, $v^n_k$ is injective on homology.

To check that $\nu_n(T_{2,2\tau+1}) = \tau+2n+1$, consider the subquotient complex $A^n_{\tau+2n+2}$. For each $1 \leq p \leq 2 \tau + 1$:
\begin{eqnarray*}
\min(i(z_p),j(z_p)-k) =&
\begin{cases}
\min(\frac{p-1}{2}, -\frac{p-1}{2} - 2n -2) & \text{ if } p \text{ is odd};
\\
\min(\frac{p}{2}, -\frac{p-2}{2} - 2n -2)  &  \text{ if } p \text{ is even}.
\end{cases}
\\
=&
\begin{cases}
\frac{p-1}{2} & \text{ if } p \text{ is odd and }  p \leq -2n-1;
\\
-\frac{p-1}{2} - 2n - 2 & \text{ if } p \text{ is odd and }  p > -2n;
\\
\frac{p}{2}  &  \text{ if } p \text{ is even and } p \leq -2n-2;
\\
 -\frac{p-2}{2} - 2n -2 &  \text{ if } p \text{ is even and } p \geq -2n.
\end{cases}
\end{eqnarray*}
Using the above, it is straightforward to check that $A^n_{\tau+2n+2}$ contains $z_p$ for $1 \leq p \leq -4n-2$ but $z_{-4n-1} \notin A^n_{\tau+2n+2}$.
Therefore, $[z_1] = 0 $ in $H_*(A^n_{\tau+2n+2})$. Thus, $\nu_n(T_{2,2\tau+1}) = \tau+2n+1$.

\item[$\bullet$] \textbf{Computation of $\nu_n(T_{2,2\tau+1})$ for $ n \leq -(\tau+1)/2$.} Consider $A^{n'}_0$ where 
\begin{eqnarray*}
n' = \begin{cases}
-\frac{\tau+1}{2} & \text{ if } \tau \text{ is odd};
\\
-\frac{\tau}{2} - 1 & \text{ if } \tau \text{ is even}.
\end{cases}
\end{eqnarray*}
For each $1 \leq p \leq 2 \tau + 1$: 
\begin{eqnarray*}
\min(i(z_p),j(z_p)-0) =&
\begin{cases}
\min(\frac{p-1}{2}, \tau -\frac{p-1}{2} ) & \text{ if } p \text{ is odd};
\\
\min(\frac{p}{2}, \tau -\frac{p-2}{2} )  &  \text{ if } p \text{ is even}.
\end{cases}
\\
=&
\begin{cases}
\frac{p-1}{2} & \text{ if } p \text{ is odd and }  p \leq \tau+1;
\\
\tau-\frac{p-1}{2}  & \text{ if } p \text{ is odd and }  p > \tau+1;
\\
\frac{p}{2}  &  \text{ if } p \text{ is even and } p \leq \tau+1;
\\
\tau -\frac{p-2}{2}  &  \text{ if } p \text{ is even and } p > \tau+1.
\end{cases}
\end{eqnarray*}
These computations show that $v^{n'}_0$ is injective on homology:
\begin{itemize}
\item
If $\tau$ is odd, $A^{n'}_0$ contains $z_p$ for $1 \leq p \leq -2n-1$ but $z_{-2n} \notin A^{n'}_0$. Similarly, for $1\leq i \leq -n-1$, $A^{n'}_0$ contains $U^{-i}z_p$ for $1 \leq p \leq -2(n+i)-1$ but $z_{-2(n+i)} \notin A^{n'}_0$. Therefore, $[U^{-i}z_1] \neq 0$ in $H_*(A^{n'}_0)$ for $0 \leq i \leq -n-1$. So $v^{n'}_0$ is injective on homology.
\item
If $\tau$ is even, $A^{n'}_0$ contains all $z_p$ for $1 \leq p \leq 2\tau+1$. Furthermore, for $1\leq i \leq -n-1$, $A^{n'}_0$ contains $U^{-i}z_p$ for $1 \leq p \leq -2(n+i)-1$ but $z_{-2(n+i)} \notin A^{n'}_0$. Therefore, $[U^{-i}z_1] \neq 0$ in $H_*(A^{n'}_0)$ for $0 \leq i \leq -n-1$. So $v^{n'}_0$ is injective on homology.
\end{itemize}
Since $\nu^{+'}(T_{2,2\tau+1}) = 0$ is a lower bound on $\nu_n(T_{2,2\tau+1})$, we conclude that $\nu_n(T_{2,2\tau+1}) = 0$ for all $n \leq -\frac{\tau+1}{2}$.
\end{itemize}
\end{proof}
\end{theorem}

\begin{proposition} If $K$ is strongly quasipositive, then $\nu_n(K) = \tau(K) = g_4(K) = g(K)$ for all positive $n$. 
\label{strongqp}
\begin{proof}
If $K$ is strongly quasipositive, then $\nu^+(K) = \tau(K) = g_4(K) = g(K)$ by \cite{Hom-Wu}. The result immediately follows since $\tau(K) \leq \nu_n(K) \leq \nu^+(K)$ for positive $n$.
\end{proof}
\end{proposition}

\example
\label{ex:T2,2N+1}
 Figure \ref{fig:T2,9}(a) shows the knot Floer chain complex $\cfk^\infty$ of the $(2, 9)$-torus knot. The computation of $\nu_{-2}(T_{2,9})$ is shown in  Figure \ref{fig:T2,9}(b)-(c). 
\begin{eqnarray*}
 \nu_n(T_{2,9}) =   \begin{cases}
 4, \text{ for all } n \geq 0 \\
 3, \text{ for } n = -1 \\
 1, \text{ for } n = -2 \\
 0, \text{ for all } n \leq -3
 \end{cases}
 \end{eqnarray*}

\begin{figure}[thb!]
\centering 
\subfigure[Generating complex for $\cfk^\infty$ of the (2, 9)-torus knot $T_{2,9}$]{%
\begin{tikzpicture}
	\draw[step=1, black!30!white, very thin] (-.9, -.9) grid (4.9, 4.9);
	
	\begin{scope}[thin, black!60!white]
		\draw [<->] (-1, 0) -- (5, 0);
		\draw [<->] (0, -1) -- (0, 5);
	\end{scope}	
	 \useasboundingbox (-0.5, -0.5) rectangle (5, 5);
        
        \filldraw (4, 0) circle (1.5pt) node[] (a){};
        \filldraw (4, 1) circle (1.5pt) node[] (b){};
        \filldraw (3, 1) circle (1.5pt) node[] (c){};
        \filldraw (3, 2) circle (1.5pt) node[] (d){};
        \filldraw (2, 2) circle (1.5pt) node[] (e){};
        \filldraw (2, 3) circle (1.5pt) node[] (f){};
        \filldraw (1, 3) circle (1.5pt) node[] (g){};
        \filldraw (1, 4) circle (1.5pt) node[] (h){};
        \filldraw (0, 4) circle (1.5pt) node[] (i){};
        \node [right] at (a) {$z_9$};
        \node [right] at (b) {$z_8$};
        \node [left] at (c) {$z_7$};
        \node [right] at (d) {$z_6$};
        \node [left] at (e) {$z_5$};
        \node [right] at (f) {$z_4$};        
        \node [left] at (g) {$z_3$};
        \node [right] at (h) {$z_2$};
        \node [left] at (i) {$z_1$};
        
	\draw [very thick, <-] (a) -- (b);
	\draw [very thick, <-] (c) -- (b);
	\draw [very thick, <-] (c) -- (d);
	\draw [very thick, <-] (e) -- (d);
	\draw [very thick, <-] (g) -- (f);
	\draw [very thick, <-] (e) -- (f);
	\draw [very thick, <-] (i) -- (h);
	\draw [very thick, <-] (g) -- (h);
\end{tikzpicture}
}
\hspace{1cm}
\subfigure[The classes $\lbrack U^{-1}z_1\rbrack$ and $\lbrack z_1\rbrack= \lbrack z_3\rbrack$ generate $\hf^2(S^3)$]{%
\begin{tikzpicture}
	\draw[step=1, black!30!white, very thin] (-.9, -.9) grid (5.5, 5.5);
	
	\begin{scope}[thin, black!60!white]
		\draw [<->] (-1, 0) -- (5.5, 0);
		\draw [<->] (0, -1) -- (0, 5.5);
	\end{scope}
	
	\filldraw[black!20!white, pattern=north west lines,] (0, -.9) rectangle (1, 5.4);
	
	 \useasboundingbox (-0.5, -0.5) rectangle (5.5, 5.5);
	 
        \filldraw (1, 5) circle (1.5pt) node[] (upper){};
        \filldraw (1, 3) circle (1.5pt) node[] (g){};
        \filldraw (1, 4) circle (1.5pt) node[] (h){};
        \filldraw (0, 4) circle (1.5pt) node[] (i){};     
        \node [right] at (g) {$z_3$};
        \node [right] at (h) {$z_2$};
        \node [left] at (i) {$z_1$};
        \node [right] at (upper) {$U^{-1}z_1$};
	\draw [very thick, <-] (i) -- (h);
	\draw [very thick, <-] (g) -- (h);
\end{tikzpicture}
}
\hspace{1cm}
\subfigure[The classes $\lbrack z_1 \rbrack$ and $\lbrack z_3 \rbrack$ vanish in $\hf^2(S^3_{-N}K, \text{[}2\text{]})$ ]{
\begin{tikzpicture}
	\draw[step=1, black!30!white, very thin] (-.4, -.4) grid (4.9, 5.4);
	
	\begin{scope}[thin, black!60!white]
		\draw [<->] (-.5, 0) -- (5, 0);
		\draw [<->] (0, -.5) -- (0, 5.5);
	\end{scope}
	
	\filldraw[black!20!white, pattern=north west lines,] (4.9, 3) rectangle (0, 2);
	\filldraw[black!20!white, pattern=north west lines,] (1, 5.4) rectangle (0, 2);

	 \useasboundingbox (-0.5, -0.5) rectangle (5.5, 5.5);
        
        \filldraw (1, 5) circle (1.5pt) node[] (upper){};
        \filldraw (3, 2) circle (1.5pt) node[] (d){};
        \filldraw (2, 2) circle (1.5pt) node[] (e){};
        \filldraw (2, 3) circle (1.5pt) node[] (f){};
        \filldraw (1, 3) circle (1.5pt) node[] (g){};
        \filldraw (1, 4) circle (1.5pt) node[] (h){};
        \filldraw (0, 4) circle (1.5pt) node[] (i){};       

        \node [right] at (d) {$z_6$};
        \node [left] at (e) {$z_5$};
        \node [right] at (f) {$z_4$};        
        \node [left] at (g) {$z_3$};
        \node [right] at (h) {$z_2$};
        \node [left] at (i) {$z_1$};
        \node [right] at (upper) {$U^{-1}z_1$};
	\draw [very thick, <-] (e) -- (d);
	\draw [very thick, <-] (g) -- (f);
	\draw [very thick, <-] (e) -- (f);
	\draw [very thick, <-] (i) -- (h);
	\draw [very thick, <-] (g) -- (h);
\end{tikzpicture}
}
\hspace{1cm}
\subfigure[The classes $\lbrack z_1 \rbrack = \lbrack z_3\rbrack$ and $\lbrack U^{-1}z_1 \rbrack$ survive in $\hf^2(S^3_{-N}K, \text{[}1\text{]})$]{
\begin{tikzpicture}
	\draw[step=1, black!30!white, very thin] (-.4, -.4) grid (4.9, 5.5);
	
	\begin{scope}[thin, black!60!white]
		\draw [<->] (-.5, 0) -- (5, 0);
		\draw [<->] (0, -.5) -- (0, 5.5);
	\end{scope}
	
	\filldraw[black!20!white, pattern=north west lines,] (4.9, 2) rectangle (0, 1);
	\filldraw[black!20!white, pattern=north west lines,] (1, 5.4) rectangle (0, 1);

	 \useasboundingbox (-0.5, -0.5) rectangle (5.5, 5.5);
        \filldraw (1, 5) circle (1.5pt) node[] (upper){};
        \filldraw (4, 1) circle (1.5pt) node[] (b){};
        \filldraw (3, 1) circle (1.5pt) node[] (c){};
        \filldraw (3, 2) circle (1.5pt) node[] (d){};
        \filldraw (2, 2) circle (1.5pt) node[] (e){};
        \filldraw (1, 3) circle (1.5pt) node[] (g){};
        \filldraw (1, 4) circle (1.5pt) node[] (h){};
        \filldraw (0, 4) circle (1.5pt) node[] (i){};     
        \node [right] at (b) {$z_8$};
        \node [left] at (c) {$z_7$};  
        \node [right] at (d) {$z_6$};
        \node [left] at (e) {$z_5$};
        \node [left] at (g) {$z_3$};
        \node [right] at (h) {$z_2$};
        \node [left] at (i) {$z_1$};
        \node [right] at (upper) {$U^{-1}z_1$};
	\draw [very thick, <-] (e) -- (d);
	\draw [very thick, <-] (c) -- (b);
	\draw [very thick, <-] (c) -- (d);
	\draw [very thick, <-] (i) -- (h);
	\draw [very thick, <-] (g) -- (h);
\end{tikzpicture}
}
\caption[]{} 
\label{fig:T2,9}
\end{figure}

The computation of $\nu_n(K)$ for thin knots show that the sequence $\nu_n$ can increase by more than one at a time, in contrast to the local $h$-invariants defined by Rasmussen, which jump by at most one (Proposition 7.6, \cite{rasmussen}). 

In fact, the gap between $\nu_n(K)$ and $\nu_{n+1}(K)$ can be arbitrarily big. For example, a straightforward (partial) computation of $\nu_n(T_{p,p+1})$ using $\cfk^\infty(T_{p,p+1})$ shows that for $p > 3$,
\[\nu_{-1}(T_{p, p+1}) - \nu_{-2}(T_{p,p+1}) = p.\]

\begin{theorem}
Let $T_{p, p+1}$ denote the $(p, p+1)$ torus knot for $p > 3$. Let $\tau = \tau(T_{p,p+1}) = \frac{(p-1)p}{2}$.
\begin{eqnarray*}
\nu_n(T_{p, p+1}) = \begin{cases}
\tau & \text{for }  n \geq 0 \\
\tau - 1 & \text{for } n = -1 \\
\tau - 1-p & \text{for } n = -2 \\
\end{cases}
\end{eqnarray*}
Thus, $\nu_{-1}(T_{p, p+1}) - \nu_{-2}(T_{p,p+1}) = p .$
\begin{proof}
In \cite{samanthaallen}, Allen shows that the staircase model chain complex for $\cfk^\infty(T_{p,p+1})$ takes the form 
\[ \lbrack 1,\ p-1, \ 2,\ p-2, \dots,\ j,\ p-j,\ \dots,\ p-1,\ 1\rbrack,\]
where the indices alternate between the widths of the horizontal and vertical steps. 
From this staircase description, there exists a $(i,j)$-filtered basis for $\cfk^\infty(T_{p,p+1})$ consisting of generators $\{b_l\}_{l=0}^{2(p-1)}$ lying in $(i,j)$-filtrations:
\begin{eqnarray*}
b_{2m}: && (\sum_{k=1}^{m}k, \frac{(p-1)p}{2}-\sum_{k=1}^{m}(p-k)),
\\
b_{2m+1}: && (\sum_{k=1}^{m+1}k, \frac{(p-1)p}{2}-\sum_{k=1}^{m}(p-k)),
\end{eqnarray*}
 and differential:
\begin{eqnarray*}
\partial b_{2m} &=& 0, 
\\
\partial b_{2m+1} &=& b_{2m}+b_{2m+2}.
\end{eqnarray*}
The same argument for showing that $\nu^+(T_{2,2\tau+1}) = \tau(T_{2,2\tau+1})$ in Theorem~\ref{thinknots} holds for the knots $T_{p,p+1}$.  Moreover, in the terminology of \cite{epsilon}, the basis $\{b_l\}_{l=0}^{2(p-1)}$ satisfies:
\begin{itemize}
\item $b_0$ is the vertically distinguished element of a vertically simplified basis.
\item $b_0$ has a unique incoming horizontal arrow (from $b_1$) (and no outgoing horizontal arrows).
\end{itemize}
We immediately conclude that $\varepsilon(T_{p,p+1}) = 1$ and $\nu_{-1}(T_{p,p+1}) = \tau - 1$.

To show $\nu_{-2}(T_{p,p+1}) = \tau-p-1$, we observe:
\begin{itemize}
\item $A_{\tau-p-1}^{-2}$ contains the generators $b_0, b_1, b_2$, but $b_3 \notin A_{\tau-p-1}^{-2}$. Therefore, $[b_0] \neq 0 $ in $H_*(A_{\tau-p-1}^{-2})$. Moreover, $[U^{-1}b_0] \neq 0 $ in $H_*(A_{\tau-p-1}^{-2})$. Thus $v_{\tau-p-1}^{-2}$ is injective on homology. 
\item $A_{\tau-p}^{-2}$ contains the generators $b_0, b_1, b_2, b_3$, but $b_4 \notin A_{\tau-p}^{-2}$. Therefore, $[b_0] = 0 $ in $H_*(A_{\tau-p}^{-2})$. 
\end{itemize}
\end{proof}
\end{theorem}

A natural question about the invariants $\nu_n(K)$ we can answer is:
\begin{question}
Do $\{\nu_n(K)\}$ contain more concordance information than the collection $\{\tau$, $\nu$, $\nu'$, $\nu^+$, $\nu^{+'}\}$? 
\end{question}
A positive answer to this question exists, illustrated in the following example. 
\example
\label{ex:T4,5}
The torus knot $T_{4,5}$ and the torus knot $T_{2, 13}$ share the following equal invariants:
\begin{eqnarray*}
\nu^{+'}(T_{4,5}) &= \ 0 \  =& \nu^{+'}(T_{2, 13}),
 \\
\nu'(T_{4,5})  &= \ 5\  = & \nu'(T_{2, 13}) ,
\\
\tau(T_{4,5})  = \nu(T_{4,5})  = \nu^+(T_{4,5})  &= \ 6\ =& \tau(T_{2, 13}) = \nu(T_{2, 13}) = \nu^+(T_{2, 13}) . 
\end{eqnarray*}
However, the invariants $\nu_n(T_{4,5})$ are different from $\nu_n(T_{2,13})$: 
\begin{equation*}
\begin{split}
\nu_n(T_{4,5}) = \begin{cases}
6, \text{ for } n \geq 0 \\
5, \text{ for } n = -1 \\
1, \text{ for } n = -2 \\
0, \text{  for } n \leq -3
\end{cases}
\end{split}
\qquad \qquad
\begin{split}
\nu_n(T_{2,13}) = \begin{cases}
6, \text{ for } n \geq 0 \\
5, \text{ for } n = -1 \\
3, \text{ for } n = -2 \\
1, \text{ for } n = -3 \\
0, \text{  for } n \leq -4
\end{cases}
\end{split}
\end{equation*}
where $\nu_n(T_{4,5})$ is computed from the knot Floer chain complex $\cfk^\infty(T_{4,5})$ as shown in Figure \ref{fig:T4,5}.
Comparing the computations of $\nu_n(K)$ for $T_{4,5}$ and $T_{2,13}$, we see that $\nu_n(K)$ contain more concordance information than the set of invariants $\{\tau$, $\nu$, $\nu'$, $\nu^+$, $\nu^{+'}\}$. 

\begin{figure}[thb!]
\centering
\begin{tikzpicture} 
	\draw[step=.5, black!30!white, very thin] (-.9, -.9) grid (3.4, 3.4);
	
	\begin{scope}[thin, black!60!white]
		\draw [<->] (-1, 0) -- (3.5, 0);
		\draw [<->] (0, -1) -- (0, 3.5);
	\end{scope}
	
	 \useasboundingbox (-0.5, -0.5) rectangle (3.5, 3.5);
        
        \filldraw (3, 0) circle (1.5pt) node[] (a){};
        \filldraw (3, .5) circle (1.5pt) node[] (b){};
        \filldraw (1.5, .5) circle (1.5pt) node[] (c){};
	\filldraw (1.5, 1.5) circle (1.5pt) node[] (d){};
        \filldraw (.5, 1.5) circle (1.5pt) node[] (e){};
        \filldraw (.5, 3) circle (1.5pt) node[] (f){};
        \filldraw (0, 3) circle (1.5pt) node[] (g){};
        \node [below] at (a) {$b_6$};
        \node [right] at (b) {$b_5$};
        \node [left] at (c) {$b_4$};
        \node [right] at (d) {$b_3$};
        \node [left] at (e) {$b_2$};
        \node [right] at (f) {$b_1$};        
        \node [left] at (g) {$b_0$};

	\draw [very thick, ->] (b) -- (a);
	\draw [very thick, ->] (b) -- (c);
	\draw [very thick, ->] (d) -- (c);
	\draw [very thick, ->] (d) -- (e);
	\draw [very thick, ->] (f) -- (g);
	\draw [very thick, ->] (f) -- (e);

\end{tikzpicture}
\caption[]
{Generating complex for $\cfk^\infty$ of the left-handed (4,5)-torus knot $T_{4,5}$} $\cfk^\infty(T_{4,5})$ is generated over $\mathbb{F}[U,U^{-1}]$ by the above chain complex. The arrows, representing terms in the differential, are drawn to scale, with lengths of arrows ranging between one and three. 
\label{fig:T4,5}
\end{figure}

\section{Further Directions} \label{section-future}
One question is the effectiveness of $\nu_n(K)$ when compared to other concordance invariants such as $\Upsilon_K(t)$, coming from $t$-modified knot Floer homology \cite{Upsilon}, or $V_k$, coming from surgery formulas of Ozsv\'ath and Szab\'o \cite{OzSzintegersurgeries}.  

The invariants $\nu_n(K)$ do not define concordance homomorphisms $\mathcal{C} \to \mathbb{Z}$, where $\mathcal{C}$ is the concordance group of knots. This implies that they do not necessarily vanish on knots of finite concordance order. 
The existence of $p$-torsion, $p \neq 2$, in the concordance group $\mathcal{C}$ is an open question. A related conjecture, based on a question of Gordon \cite{gordon}, as phrased in \cite{livsurvey}:
\begin{conjecture}[Gordon]
A knot has order two in $\mathcal{C}$ if and only if $K = -K$ is negative amphichiral.
\end{conjecture} 
Recently, Hendricks and Manolescu defined involutive Heegaard Floer concordance invariants $\overline{V_0}$ and $\underline{V_0}$, which detects the non-sliceness of the figure eight knot. The non-sliceness of $4_1$ was previously known through classical methods, but this is the first method of detection coming from the Heegaard Floer package. By additivity of $\tau$ and the behavior of $\varepsilon$ under connect sums, $\tau(K)$ and $\nu(K)$ vanish for all knots $K$ of finite concordance order. This leaves open the cases $\nu_n(K)$ for $n > 1$ and $n < -1$. 
We pose the question:
\begin{question}
Does there exist a knot $K$ of finite concordance order such that $\nu_n(K) \neq 0$ for some $n$?
\end{question}

Another question is how the invariants $\nu_n(K)$ behave under connected sum. It is known that $\nu^+(K)$ is subadditive by \cite{BCGnote}. That is, 
\begin{eqnarray*}
\nu^+(K\#L) \leq \nu^+(K) + \nu^+(L)
\end{eqnarray*}
Using mirroring relations and subadditivity of $\nu^+(K)$ shows that $\nu^{+'}(K)$ is superadditive.
\begin{lemma} 
\label{lem-superadditive}
For any two knots $K$ and $L$, 
\begin{eqnarray*}
\nu^{+'}(K\#L) \geq \nu^{+'}(K) + \nu^{+'}(L)
\end{eqnarray*}
\begin{proof} By subadditivity of $\nu^+$ and the mirroring relations,
\begin{eqnarray*}
\nu^+(-K\#-L) &\leq& \nu^+(-K) + \nu^+(-L)
\\
-\nu^{+'}(K\#L) &\leq& -\nu^{+'}(K) + -\nu^{+'}(L)
\\
\nu^{+'}(K\#L) &\geq& \nu^{+'}(K) + \nu^{+'}(L)
\end{eqnarray*}
\end{proof}
\end{lemma}
As pointed out to the author by Jen Hom, it can also be seen by additivity of $\tau$ and the behavior of $\varepsilon$ under connect sum that $\nu(K)$ is subadditive. A similar argument shows that $\nu'(K)$ is superadditive. This leads us to ask the following two questions. 

\begin{question}
Is $\nu_n(K \# K') \leq \nu_n(K) + \nu_n(K')$ for all positive integers $n \in \mathbb{Z}_+$?
\end{question}
\begin{question}
Is $\nu_n(K \# K') \geq \nu_n(K) + \nu_n(K')$ for all negative integers $n \in \mathbb{Z}_-$?
\end{question}

The next question was posed by Zhongtao Wu:
\begin{question}[Wu]
If $\nu_n(K) = \nu_n(K')$ for all $n \in \mathbb{Z}$, then is $\nu^+(K \# -K')
 = \nu^+(-K \# K') = 0$?
 \end{question}
The condition that $\nu^+(K \# -K') = \nu^+(-K \# K') = 0$ implies that \[\cfk^\infty(K \# -K') \simeq \cfk^\infty(U) \oplus A,\] where $A$ is an acyclic complex \cite{hom-survey}. 

 \bibliographystyle{amsplain}
 \bibliography{references}

\providecommand{\bysame}{\leavevmode\hbox to3em{\hrulefill}\thinspace}
\providecommand{\MR}{\relax\ifhmode\unskip\space\fi MR }
\providecommand{\MRhref}[2]{%
  \href{http://www.ams.org/mathscinet-getitem?mr=#1}{#2}
}
\providecommand{\href}[2]{#2}
\begin{thebibliography}{10}

\bibitem{samanthaallen}
Samantha {Allen}, \emph{{Using secondary Upsilon invariants to rule out stable
  equivalence of knot complexes}}, arXiv:1706.07108 (2017).

\bibitem{BCGnote}
J\'ozsef Bodn\'ar, Daniele Celoria, and Marco Golla, \emph{A note on cobordisms
  of algebraic knots}, Algebr. Geom. Topol. \textbf{17} (2017), no.~4,
  2543--2564. \MR{3686406}

\bibitem{gordon}
C.~McA. Gordon, \emph{Some aspects of classical knot theory}, Knot theory
  ({P}roc. {S}em., {P}lans-sur-{B}ex, 1977), Lecture Notes in Math., vol. 685,
  Springer, Berlin, 1978, pp.~1--60. \MR{521730}

\bibitem{hedden-positivity}
Matthew Hedden, \emph{Notions of positivity and the {O}zsv\'ath-{S}zab\'o
  concordance invariant}, J. Knot Theory Ramifications \textbf{19} (2010),
  no.~5, 617--629. \MR{2646650}

\bibitem{epsilon}
Jennifer Hom, \emph{Bordered {H}eegaard {F}loer homology and the tau-invariant
  of cable knots}, J. Topol. \textbf{7} (2014), no.~2, 287--326. \MR{3217622}

\bibitem{hom-survey}
\bysame, \emph{A survey on {H}eegaard {F}loer homology and concordance}, J.
  Knot Theory Ramifications \textbf{26} (2017), no.~2, 1740015, 24.
  \MR{3604497}

\bibitem{Hom-Wu}
Jennifer Hom and Zhongtao Wu, \emph{Four-ball genus bounds and a refinement of
  the {O}zv\'ath-{S}zab\'o tau invariant}, J. Symplectic Geom. \textbf{14}
  (2016), no.~1, 305--323. \MR{3523259}

\bibitem{liv04}
Charles Livingston, \emph{Computations of the {O}zsv\'ath-{S}zab\'o knot
  concordance invariant}, Geom. Topol. \textbf{8} (2004), 735--742.
  \MR{2057779}

\bibitem{livsurvey}
\bysame, \emph{A survey of classical knot concordance}, Handbook of knot
  theory, Elsevier B. V., Amsterdam, 2005, pp.~319--347. \MR{2179265}

\bibitem{verticaltruncation}
Ciprian {Manolescu} and Peter {Ozsv{\'a}th}, \emph{{Heegaard Floer homology and
  integer surgeries on links}}, arXiv:1011.1317 (2010).

\bibitem{HFK}
Peter Ozsv{\'a}th and Zolt{\'a}n Szab{\'o}, \emph{Holomorphic disks and knot
  invariants}, Adv. Math. \textbf{186} (2004), no.~1, 58--116. \MR{2065507
  (2005e:57044)}

\bibitem{OzSzfour}
\bysame, \emph{Holomorphic triangles and invariants for smooth four-manifolds},
  Adv. Math. \textbf{202} (2006), no.~2, 326--400. \MR{2222356 (2007i:57029)}

\bibitem{OzSzintegersurgeries}
\bysame, \emph{Knot {F}loer homology and integer surgeries}, Algebr. Geom.
  Topol. \textbf{8} (2008), no.~1, 101--153. \MR{2377279}

\bibitem{OzSzRational}
\bysame, \emph{Knot {F}loer homology and rational surgeries}, Algebr. Geom.
  Topol. \textbf{11} (2011), no.~1, 1--68. \MR{2764036}

\bibitem{Upsilon}
Peter~S. Ozsv\'ath, Andr\'as~I. Stipsicz, and Zolt\'an Szab\'o,
  \emph{Concordance homomorphisms from knot {F}loer homology}, Adv. Math.
  \textbf{315} (2017), 366--426. \MR{3667589}

\bibitem{Petkova}
Ina Petkova, \emph{Cables of thin knots and bordered {H}eegaard {F}loer
  homology}, Quantum Topol. \textbf{4} (2013), no.~4, 377--409. \MR{3134023}

\bibitem{rasmussen}
Jacob Rasmussen, \emph{Floer homology and knot complements}, ProQuest LLC, Ann
  Arbor, MI, 2003, Thesis (Ph.D.)--Harvard University. \MR{2704683}

\end{thebibliography}

\end{document}